\documentclass[oneside,a4paper]{amsart}
\usepackage{amsmath,amssymb,amsthm,longtable,mathrsfs,amscd,multirow,latexsym}
\usepackage[dvipdfmx]{graphicx}
\usepackage{longtable}
\usepackage{stmaryrd}
\usepackage{tikz}
\usetikzlibrary{calc}

\usepackage[top=30truemm,bottom=30truemm,left=25truemm,right=25truemm]{geometry}

\newtheorem{dfn}{Definition}[section]
\newtheorem{thm}{Theorem}[section]
\newtheorem{prp}[thm]{Proposition}
\newtheorem{lmm}[thm]{Lemma}

\newtheorem{rem}[thm]{Remark}

\makeatletter
    
    \@addtoreset{equation}{section}
\makeatother
\makeatletter
  \newcommand{\subsubsubsection}{\@startsection{paragraph}{4}{\z@}%
    {1.0\Cvs \@plus.5\Cdp \@minus.2\Cdp}%
    {.1\Cvs \@plus.3\Cdp}%
    {\reset@font\sffamily\normalsize}
  }
  \makeatother
\allowdisplaybreaks[4]

\newcommand{\idn}{\mathop{\text{\rm Id}}}
\newcommand{\supp}{\mathop{\text{\rm supp}}}
\newcommand{\chr}{\mathop{\text{\rm ch}}}

\newcommand{\rs}{\circ}
\newcommand{\pl}{\prec}
\newcommand{\pr}{\succ}
\newcommand{\pge}{\succeq}

\newcommand{\tri}{|\!|\!|}
\newcommand{\brc}[1]{[\![#1]\!]}
\newcommand{\ovl}[1]{\overline{#1}}

\def\bsf{\boldsymbol{f}}
\def\bsg{\boldsymbol{g}}

\title{Commutator estimates from a viewpoint of regularity structures}

\author{Masato Hoshino}
\address{Faculty of Mathematics, Kyushu University}
\email{hoshino@math.kyushu-u.ac.jp}

\date{}

\begin{document}

\maketitle

\begin{abstract}      
First we introduce the Bailleul-Hoshino's result \cite{BH18}, which links the theory of regularity structures and the paracontrolled calculus. As an application of their result, we give another algebraic proof of the multicomponent commutator estimate \cite{BB16}, which is a generalized version of the Gubinelli-Imkeller-Perkowski's commutator estimate \cite[Lemma~2.4]{GIP15}.
\end{abstract}

\section{Introduction}

In this paper, we introduce the recent research by Bailleul and Hoshino \cite{BH18} and show its application to the \emph{commutator estimate}, one of the important tools in the analysis of singular SPDEs.

Singular SPDEs often involve ill-defined products of distributions.
The theory of \emph{regularity structures} by Hairer \cite{Hai14} and the \emph{paracontrolled calculus} by Gubinelli, Imkeller and Perkowski \cite{GIP15} provide general approaches to give a meaning to such SPDEs.
Both of them are extensions of the \emph{rough path theory}, which was originally introduced by Lyons \cite{Lyo98} and reformulated by Gubinelli \cite{Gub04}. The latter version consists of a Lie group of enhanced noises (called \emph{rough paths}) and a fiber bundle of enhanced solution spaces (called \emph{controlled paths}). The \emph{It\^o-Lyons map} is the continuous mapping from a given rough path to the solution of SDE in the class of controlled paths. This part is purely deterministic. The only probabilistic part is how to lift the stochastic noise (typically, a Brownian motion) to the rough path.
Both of the regularity structures and the paracontrolled calculus share the same spirit.
The only difference between them is the definition of enhanced noises and enhanced solutions.
The rough meaning of this difference is whether the function/distribution is described \emph{locally} or \emph{globally}. For example, consider the typical SDE
$$
dX_t=f(X_t)dB_t.
$$
In the theory of regularity structures, its solution $X$ is assumed to have the local structure
$$
X_t-X_s=X_s'(B_t-B_s)+O(|t-s|^{1-}).
$$
(This is nothing but the definition of controlled path.) Here the global description of the solution is not given, so we need a \emph{reconstruction operator} explained later. In the paracontrolled calculus, the solution $X$ is assumed to have the global structure
$$
X=X'\pl B+(\mathcal{C}^{1-}),
$$
where $\prec:\mathcal{S}'\times\mathcal{S}'\to\mathcal{S}'$ is a bilinear continuous operator defined below, called \emph{paraproduct}.
These two kinds of definitions have their own merits. In the regularity structures, the way to give a meaning to singular SPDEs is automated in \cite{BHZ16, CH16, BCCH17} by using the language of Hopf algebra. In the paracontrolled calculus, since the solution is given globally, we can study more detailed properties of specific singular SPDEs \cite{MW17, AK17, GH18, Hos18} by using well-known techniques in the real analysis. Hence we can use either of them according to the situation. 

The equivalence between the two theories is not well studied. One of the studies is in the Gubinelli-Imkeller-Perkowski's original paper \cite{GIP15}. In its last section, the authors showed that the reconstruction operator is represented by an integral operator like a paraproduct.
By using such operator, Martin and Perkowski \cite{MP18} introduced an intermediate notion between the local and global descriptions as above and show the equivalence between the local and intermediate forms.
However, these studies still depend on the local form of the solution.
Bailleul and Hoshino \cite{BH18} introduced the ``paracontrolled remainders" and showed the paracontrolled form of the reconstruction operator. They also showed a partly equivalent relation between the local and global forms.

In this paper, first we briefly introduce their result and next apply it to the proof of the \emph{commutator estimate} \cite[Lemma~2.4]{GIP15}. Their commutator is defined by
$$
C(f,g,h)=(f\pl g)\rs h-f(g\rs h)
$$
for smooth inputs $(f,g,h)$. The commutator estimate implies that $C$ is uniquely extended to the continuous trilinear operator from $\mathcal{C}^\alpha\times\mathcal{C}^\beta\times\mathcal{C}^\gamma$ to $\mathcal{C}^{\alpha+\beta+\gamma}$, with $\alpha\in(0,1)$, $\beta+\gamma<0$, and $\alpha+\beta+\gamma>0$. In the above SDE, such estimate shows that the product of $f(X)$ and $\dot{B}=\frac{dB}{dt}$ can be decomposed as follows.
$$
f(X)\dot{B}=f'(X)X'(B\rs\dot{B})+(\text{continuous function of $(X,X',B,\dot{B})$}).
$$
Hence we have only to define the explicit distribution $B\rs\dot{B}$ to define the product $f(X)\dot{B}$, even though there is an unknown function $X$. This is the key point in the Fourier approach to the rough path theory \cite{GIP15}. In the analysis of SPDEs, we often need more iterated versions
\begin{align*}
C(f_1,f_2,f_3,h)&=C(f_1\pl f_2,f_3,h)-f_1C(f_2,f_3,h),\\
C(f_1,f_2,f_3,f_4,h)&=C(f_1\pl f_2,f_3,f_4,h)-f_1C(f_2,f_3,f_4,h),\dots.
\end{align*}
Bailleul and Bernicot \cite{BB16} studied many kinds of such operators.
As in \cite{GIP15, BB16}, the commutator estimate is usually proved by the Fourier analysis technique, but in this paper we show a different type of proof as an application of the Bailleul-Hoshino's result. Our approach is more algebraic and automatic than the direct computation. Moreover, it would make it clear the role of the ``paracontrolled remainder" in \cite{BH18}. I conjecture that other kinds of operators in \cite{BB16} can also be reformulated from an algebraic viewpoint.

This paper is organized as follows.
In Section \ref{2 section preliminaries}, we recall some notions and facts from the regularity structures and the paracontrolled calculus and introduce the Bailleul-Hoshino's result.
In Section \ref{3 section local behavior}, as a preparation for Section \ref{4 section commutator}, we construct a canonical Hopf algebra associated with iterated paraproducts.
In Section \ref{4 section commutator}, we show the algebraic proof of the multicomponent commutator estimate.

\section{Preliminaries and Bailleul-Hoshino's result}\label{2 section preliminaries}

From now on, we consider the functions and distributions on $\mathbb{R}^d$. Denote by $\mathcal{S}'(\mathbb{R}^d)$ the space of tempered distributions.

\subsection{Besov space and paraproduct}

We recall the Littlewood-Paley theory.
Fix a smooth radial functions $\chi$ and $\rho$ on $\mathbb{R}^d$ such that,
\begin{itemize}
\item $\supp(\chi)\subset\{x;|x|<\frac43\}$ and $\supp(\rho)\subset\{x;\frac34<|x|<\frac83\}$,
\item $\chi(x)+\sum_{j=0}\rho(2^{-j}x)=1$ for any $x\in\mathbb{R}^d$.
\end{itemize}
Set $\rho_{-1}:=\chi$ and $\rho_j:=\rho(2^{-j}\cdot)$ for $j\ge0$. We define the Littlewood-Paley blocks
$$
\Delta_jf:=\mathcal{F}^{-1}(\rho_j\mathcal{F}f),\quad
f\in\mathcal{S}'(\mathbb{R}^d),
$$
where $\mathcal{F}$ is the Fourier transform on $\mathbb{R}^d$ and $\mathcal{F}^{-1}$ is the inverse transform.
For $\alpha\in\mathbb{R}$, we define the (nonhomogeneous) Besov space
$$
\mathcal{C}^\alpha:=\{f\in\mathcal{S}'(\mathbb{R}^d)\ ;\,\|f\|_\alpha:=\sup_{j\ge-1}2^{j\alpha}\|\Delta_jf\|_{L^\infty}<\infty\}.
$$

For any smooth functions $f,g$ on $\mathbb{R}^d$, we decompose the product $fg$ as follows.
\begin{align*}
fg&=\sum_{j,k\ge-1}\Delta_jf\Delta_kg\\
&=\sum_{j<k-1}\Delta_jf\Delta_kg+\sum_{|j-k|\le1}\Delta_jf\Delta_kg+\sum_{j+1<k}\Delta_jf\Delta_kg\\
&=:f\pl g+f\rs g+f\pr g.
\end{align*}
$f\pl g=g\pr f$ is called a \emph{paraproduct}, and $f\rs g$ is called a \emph{resonant}. The following estimates are basic.

\begin{prp}[\cite{Bon81}]
Let $\alpha,\beta\in\mathbb{R}$.
\begin{enumerate}
\item If $\alpha\neq0$, then the map $\mathcal{C}^{\alpha}\times\mathcal{C}^\beta\ni(f,g)\mapsto f\pl g\in\mathcal{C}^{\alpha\wedge0+\beta}$ is well-defined and continuous.
\item If $\alpha+\beta>0$, then the map $\mathcal{C}^{\alpha}\times\mathcal{C}^\beta\ni(f,g)\mapsto f\rs g\in\mathcal{C}^{\alpha+\beta}$ is well-defined and continuous.
\end{enumerate}
\end{prp}

\subsection{Regularity structures}

We recall some important notions of the theory of regularity structures \cite{Hai14}.

\begin{dfn}\label{2 dfn crs}
A \emph{(concrete) regularity structure} $(T^+,T)$ is a pair of graded vector spaces
$$
T^+=\bigoplus_{\alpha\in A^+}T_\alpha^+,\quad
T=\bigoplus_{\beta\in A}T_\beta,
$$
where each $T_\alpha^+$ and $T_\beta$ are finite dimensional spaces and such that,
\begin{itemize}
\item $A^+,A\subset \mathbb{R}$ are countable sets bounded from below and without any accumulation point. In particular, $0=\min A^+$ and $A^++A^+\subset A^+$,
\item $T^+$ is a graded algebra ($T_{\alpha_1}^+T_{\alpha_2}^+\subset T_{\alpha_1+\alpha_2}^+$) with unit $\mathbf{1}$ and $T_0^+=\mathbb{R}\mathbf{1}$,
\item $T^+$ is a graded Hopf algebra with coproduct $\Delta^+:T^+\to T^+\otimes T^+$ such that, $\Delta^+\mathbf{1}=\mathbf{1}\otimes\mathbf{1}$ and
$$
\Delta^+\tau\in\tau\otimes\mathbf{1}+\mathbf{1}\otimes\tau+\bigoplus_{0<\beta<\alpha}T_\beta^+\otimes T_{\alpha-\beta}^+
$$
for any $\tau\in T_\alpha^+$ with $\alpha>0$,
\item $T$ has a coproduct $\Delta:T\to T\otimes T^+$ with the right comodule property $(\Delta\otimes\idn)\Delta=(\idn\otimes\Delta^+)\Delta$ and with
$$
\Delta\tau\in\tau\otimes\mathbf{1}+\bigoplus_{\beta<\alpha}T_\beta\otimes T_{\alpha-\beta}^+
$$
for $\tau\in T_\alpha$.
\end{itemize}
\end{dfn}

By definition, there exists $\min A\in\mathbb{R}$, which is called a \emph{regularity} of $T$. From now on, we set $\alpha_0:=(\min A)\wedge0$.
We denote by $\|\cdot\|_\alpha$ the equivalent norm on the finite dimensional space $T_\alpha$.
For an arbitrary $\tau\in T$, we write $\|\tau\|_\alpha$ for the norm of the projection of $\tau$ into $T_\alpha$.
We fix the bases $\mathcal{B}_\alpha^+$ and $\mathcal{B}_\beta$ of $T_\alpha^+$ and $T_\beta$, respectively. We set $\mathcal{B}^+=\bigcup_{\alpha\in A^+}\mathcal{B}_\alpha^+$ and $\mathcal{B}=\bigcup_{\beta \in A}\mathcal{B}_\beta$. For any $\tau\in\mathcal{B}_\alpha^+$ and $\sigma\in\mathcal{B}_\beta$, we write $|\tau|=\alpha$ and $|\sigma|=\beta$.

Next we define models on $(T^+,T)$. Note that the space of all nonzero algebra homomorphisms $f:T^+\to\mathbb{R}$ forms a \emph{character group} $G=\chr(T^+)$ by the product
$$
f*g:=(f\otimes g)\Delta^+.
$$
We define the class of test functions
$$
\Phi:=\left\{\varphi:\mathbb{R}^d\to\mathbb{R}\ ;
\begin{aligned}
&\bullet\ \text{$\|\varphi\|_{C^r}:=\sup_{|k|\le r}\|\partial_x^k\varphi\|_{L^\infty}\le1$ for an integer $r>-\alpha_0$},\\
&\bullet\ \text{$\supp(\varphi)\subset\{x;|x|\le1\}$}.
\end{aligned}
\right\}.
$$
Given $\varphi\in\Phi$, $x\in\mathbb{R}^d$, and $0<\lambda\le1$, we set $\varphi_x^\lambda(\cdot):=\lambda^{-d}\varphi(\lambda^{-1}(\cdot-x))$.

\begin{dfn}\label{2 dfn model}
For a function ${\sf g}:\mathbb{R}^d\to G$ and a linear map ${\sf \Pi}:T\to\mathcal{C}^{\alpha_0}$, we set
$$
{\sf g}_{yx}:=({\sf g}_y\otimes{\sf g}_x^{-1})\Delta^+,\quad
{\sf\Pi}_x^{\sf g}:=({\sf\Pi}\otimes{\sf g}_x^{-1})\Delta.
$$
A model ${\sf Z}=({\sf g},{\sf\Pi})$ is a pair of such functions ${\sf g}$ and ${\sf \Pi}$ such that,
\begin{align*}
\|{\sf g}\|&:=\sup_{\tau\in\mathcal{B}^+}\sup_{x\in\mathbb{R}^d}|{\sf g}_x(\tau)|+
\sup_{\tau\in\mathcal{B}^+}\sup_{x,y\in\mathbb{R}^d}\frac{|{\sf g}_{yx}(\tau)|}{|y-x|^{|\tau|}}<\infty,\\
\|{\sf\Pi}\|^{\sf g}&:=\sup_{\tau\in\mathcal{B}}\|{\sf\Pi}\tau\|_{\alpha_0}+
\sup_{\tau\in\mathcal{B}}\sup_{x\in\mathbb{R}^d}\sup_{\varphi\in\Phi}\sup_{0<\lambda\le1}\lambda^{-|\tau|}|\langle{\sf\Pi}_x^{\sf g}\tau,\varphi_x^\lambda\rangle|<\infty.
\end{align*}
We set $\tri{\sf Z}\tri:=\|{\sf g}\|+\|{\sf\Pi}\|^{\sf g}$. Although the set of all models is not linear, we can define the metric $d({\sf Z},{\sf Z}')=\tri{\sf Z}-{\sf Z}'\tri$ on it, by replacing ${\sf g},{\sf\Pi},{\sf\Pi}_x^{\sf g}$ by ${\sf g}-{\sf g}',{\sf\Pi}-{\sf\Pi}',{\sf\Pi}_x^{\sf g}-({\sf\Pi}')_x^{{\sf g}'}$ in the above definition.
\end{dfn}

We define the class of modelled distributions.

\begin{dfn}
For a model ${\sf Z}=({\sf g},{\sf\Pi})$, we define the operator on $T$ by
$$
\hat{\sf g}_{yx}:=(\idn\otimes{\sf g}_{yx})\Delta.
$$
Let $\gamma\in\mathbb{R}$. A function $\bsf:\mathbb{R}^d\to T_{<\gamma}:=\bigoplus_{\alpha<\gamma}T_\alpha$ is called a \emph{$\gamma$-class modelled distribution} if
$$
\tri\bsf\tri_\gamma
:=\sup_{\alpha<\gamma}\sup_{x\in\mathbb{R}^d}\|\bsf(x)\|_\alpha
+\sup_{\alpha<\gamma}\sup_{x,y\in\mathbb{R}^d}
\frac{\|\bsf(y)-\hat{\sf g}_{yx}\bsf(x)\|_\alpha}{|y-x|^{\gamma-\alpha}}<\infty.
$$
Denote by $\mathcal{D}^\gamma({\sf g})$ the space of all $\gamma$-class modelled distributions.
\end{dfn}

For two elements $\bsf\in\mathcal{D}^\gamma({\sf g})$ and $\bsf'\in\mathcal{D}^\gamma({\sf g}')$ modelled by different models ${\sf g}$ and ${\sf g}'$ respectively, we define the quantity $d_\gamma(\bsf,\bsf')$ by replacing $\bsf(x)$ and $\bsf(y)-\hat{\sf g}_{yx}\bsf(x)$ by $\bsf(x)-\bsf'(x)$ and $(\bsf(y)-\hat{\sf g}_{yx}\bsf(x))-(\bsf'(y)-\hat{\sf g}_{yx}'\bsf'(x))$ respectively, in the above definition.

The following theorem is the so-called reconstruction theorem.

\begin{thm}[{\cite[Theorem~3.10]{Hai14}}]\label{2 thm original reconst}
Let $\gamma>0$. For any model ${\sf Z}$, there exists a continuous operator ${\sf R}^{\sf Z}:\mathcal{D}^\gamma({\sf g})\to\mathcal{C}^{\alpha_0}$ uniquely determined by the following property. For any $\bsf\in\mathcal{D}^\gamma({\sf g})$ and $\lambda\in(0,1]$, one has
$$
\sup_{x\in\mathbb{R}^d}\sup_{\varphi\in\Phi}|({\sf R}^{\sf Z}\bsf-{\sf\Pi}_x^{\sf g}\bsf(x))(\varphi_x^\lambda)|
\lesssim\|{\sf\Pi}\|^{\sf g}\tri\bsf\tri_\gamma\lambda^\gamma.
$$
Moreover, the mapping $({\sf Z},\bsf)\mapsto{\sf R}^{\sf Z}\bsf\in\mathcal{C}^{\alpha_0}$ is continuous with respect to the metric $\tilde{d}(({\sf Z},\bsf),({\sf Z}',\bsf'))=d({\sf Z},{\sf Z}')+d_\gamma(\bsf,\bsf')$.
\end{thm}

\subsection{Bailleul-Hoshino's result}

The Bailleul-Hoshino's result \cite{BH18} consists of ``from local to global" part and ``from global to local" part. The first part provides the transformation of models into paracontrolled remainders. For any $\tau,\sigma\in\mathcal{B}$ (or $\mathcal{B}^+$), we define $\tau/\sigma\in T^+$ by the formula
$$
\Delta\tau\ (\text{or $\Delta^+\tau$})\
=\sum_{\sigma\in\mathcal{B}\ (\text{or $\mathcal{B}^+$})}\sigma\otimes(\tau/\sigma).
$$
We write $\sigma<\tau$ if $\tau/\sigma\neq0$ and $\sigma\neq\tau$. By definition, $\sigma<\tau$ implies $|\sigma|<|\tau|$.

\begin{thm}[{\cite[Proposition~11]{BH18}}]\label{2 thm BH result 1}
For any model ${\sf Z}=({\sf g},{\sf \Pi})$, we define the family of functions (or distributions) $\{\brc{\tau}^{\sf g}\}_{\tau\in\mathcal{B}^+}$ and $\{\brc{\sigma}^{\sf Z}\}_{\sigma\in\mathcal{B}}$ by the formulas
\begin{align*}
{\sf g}(\tau)&=\sum_{\nu\in\mathcal{B}^+;\mathbf{1}<\nu<\tau}{\sf g}(\tau/\nu)\pl\brc{\nu}^{\sf g}+\brc{\tau}^{\sf g},\quad\tau\in\mathcal{B}^+,\\
{\sf\Pi}\sigma&=\sum_{\mu\in\mathcal{B};\mu<\sigma}{\sf g}(\sigma/\mu)\pl\brc{\mu}^{\sf Z}+\brc{\sigma}^{\sf Z},\quad\sigma\in\mathcal{B}.
\end{align*}
Then one has
$$
\brc{\tau}^{\sf g}\in\mathcal{C}^{|\tau|},\quad
\brc{\sigma}^{\sf Z}\in\mathcal{C}^{|\sigma|},
$$
for any $\tau\in\mathcal{B}^+$ and $\sigma\in\mathcal{B}$. Moreover, the mappings ${\sf g}\mapsto\brc{\tau}^{\sf g}\in\mathcal{C}^{|\tau|}$ and ${\sf Z}\mapsto\brc{\sigma}^{\sf Z}\in\mathcal{C}^{|\sigma|}$ are continuous.
\end{thm}

They also showed that the reconstruction has the paracontrolled form.

\begin{thm}[{\cite[Theorem~13]{BH18}}]\label{2 thm reconst}
Let $\gamma>0$. For any model ${\sf Z}$, there exists a continuous operator $\brc{\cdot}^{\sf Z}:\mathcal{D}^\gamma({\sf g})\to\mathcal{C}^{\gamma}$ such that, for any $\bsf=\sum_{\tau\in\mathcal{B};|\tau|<\gamma}f_\tau\tau\in\mathcal{D}^\gamma({\sf g})$,
\begin{align*}
{\sf R}^{\sf Z}\bsf=\sum_{|\tau|<\gamma}f_\tau\pl\brc{\tau}^{\sf Z}+\brc{\bsf}^{\sf Z}.
\end{align*}
Moreover, the mapping $({\sf Z},\bsf)\mapsto\brc{\bsf}^{\sf Z}\in\mathcal{C}^\gamma$ is continuous.
\end{thm}

The second part of \cite{BH18} shows how to recover the model from a given family of paracontrolled remainders. Set $\mathcal{B}_-:=\{\tau\in\mathcal{B};|\tau|\le0\}$

\begin{prp}[{\cite[Corollary~14]{BH18}}]\label{2 prp brc to model}
Assume that a function ${\sf g}:\mathbb{R}^d\to G$ with $\|{\sf g}\|<\infty$ is given.
Then for any given family $\{\brc{\tau}\in \mathcal{C}^{|\tau|}\}_{\tau\in\mathcal{B}_-}$, there exists a unique model ${\sf\Pi}$ determined by the formula
$$
{\sf\Pi}\tau=\sum_{\sigma\in\mathcal{B};\sigma<\tau}{\sf g}(\tau/\sigma)\pl\brc{\sigma}+\brc{\tau},\quad\tau\in\mathcal{B}_-.
$$
Moreover, the mapping $({\sf g},\{\brc{\tau}\}_{\tau\in\mathcal{B}_-})\mapsto{\sf\Pi}$ is continuous.
\end{prp}

\section{Local behaviors of paraproducts}\label{3 section local behavior}

For a sequence $f_1,f_2,\dots$ of distributions, we define the \emph{iterated paraproducts}
$$
(f_1)^\pl:=f_1,\quad
(f_1,\dots,f_n)^\pl:=(f_1,\dots,f_{n-1})^\pl\pl f_n.
$$
Obviously, $(f,g)^\pl=f\pl g$.
The aim of this section is to show the following local behavior of iterated paraproduct, which has an important role in the next section.

\begin{thm}\label{3 thm main}
Let $\alpha_1,\dots,\alpha_n\in(0,1)$ and $f_i\in\mathcal{C}^{\alpha_i}$, $i=1,\dots,n$. Inductively define
\begin{align*}
\omega_{yx}^\prec(f_1,\dots,f_n)&:=(f_1,\dots,f_n)^\prec(y)-(f_1,\dots,f_n)^\prec(x)\\
&\quad-\sum_{\ell=1}^{n-1}(f_1,\dots,f_\ell)^\prec(x)\omega_{yx}^\prec(f_{\ell+1},\dots,f_n),
\quad x,y\in\mathbb{R}^d.
\end{align*}
If $\alpha_1+\cdots+\alpha_n<1$, then one has the bound
$$
|\omega_{yx}^\prec(f_1,\dots,f_n)|
\lesssim\|f_1\|_{\alpha_1}\cdots\|f_n\|_{\alpha_n}|y-x|^{\alpha_1+\cdots+\alpha_n}.
$$
\end{thm}

We also construct a concrete regularity structure and show the canonical way to lift the iterated paraproduct to such an algebraic structure.

\subsection{Word Hopf algebra and models}\label{3 section hopf}

We define the usual \emph{word Hopf algebra} as follows.
Let $S=\{1,\dots,n\}$ be the set of ``alphabets", and let $W=\bigcup_{k=0}^\infty S^k$ be the set of all ``words", i.e., finite sequences of elements of $S$. The set $W$ includes the empty word $\mathbf{1}=\emptyset$.
Let $\mathcal{W}=\mathbb{R}[W]$ be the polynomial ring, i.e., the commutative algebra finitely generated by $W$. The empty word $\mathbf{1}$ has a role of the unit.
Moreover, $\mathcal{W}$ has a coassociative coproduct $\Delta:\mathcal{W}\to\mathcal{W}\otimes\mathcal{W}$ defined by
$$
\Delta({i_1}\dots{i_k})
:=\sum_{\ell=0}^{k}({i_{\ell+1}}\dots{i_k})\otimes({i_1}\dots{i_\ell}),\quad
(i_1\dots i_k)\in W.
$$
(We understand ``$(i_1\dots i_0)$" and ``$(i_{k+1}\dots i_k)$" as the empty word $\mathbf{1}$.)
Then $\mathcal{W}$ has a Hopf algebra structure. Its counit $\mathbf{1}^*:\mathcal{W}\to\mathbb{R}$ is given by
$$
\mathbf{1}^*(\tau):=
\begin{cases}
1&\tau=\mathbf{1},\\
0&\text{$\tau$ is a word of length $\ge1$}.
\end{cases}
$$

We define the graded structure of $\mathcal{W}$.
Assume that each alphabet $i\in S$ has a homogeneity $\alpha_i\in(0,1)$. Then each word has a homogeneity
$$
|\mathbf{1}|:=0,\quad
|({i_1}\dots{i_k})|:=\alpha_{i_1}+\cdots+\alpha_{i_k}.
$$
For a product $\tau_1\cdots\tau_m$ of words, we define the homogeneity $|\tau_1\cdots\tau_m|:=|\tau_1|+\cdots+|\tau_m|$. Then we can see that $\mathcal{W}$ has a structure of graded Hopf algebra.

In this paper, we consider the subalgebra $\mathcal{W}^{<1}$ generated by $W^{<1}$, the set of all words with homogeneity $<1$.

\begin{prp}
The pair $(\mathcal{W}^{<1},\mathcal{W}^{<1})$ is a concrete regularity structure.
\end{prp}

\begin{rem}
To consider words with homogeneities $\ge1$, we need more structures. Indeed, if $\alpha>0$, then for any $f\in\mathcal{C}^\alpha$ we have
$$
f(y)=\sum_{|k|<\alpha}\frac{\partial_x^kf(x)}{k!}(y-x)^k+O(|y-x|^{\alpha}).
$$
In this case, we need additional structures associated with ``polynomials" and ``derivatives". We do not consider such structures in this paper.
\end{rem}

We consider models on the regularity structure $(\mathcal{W}^{<1},\mathcal{W}^{<1})$.
A function ${\sf g}:\mathbb{R}^d\to G:=\chr(\mathcal{W}^{<1})$ is called a model on $\mathcal{W}^{<1}$, if the pair $({\sf g},{\sf g})$ is a model on $(\mathcal{W}^{<1},\mathcal{W}^{<1})$.

\begin{lmm}
For a model ${\sf g}:\mathbb{R}^d\to G$, the operator ${\sf g}_{yx}$ is given by the formula
\begin{align*}
{\sf g}_{yx}(i_1\dots i_k)={\sf g}_y(i_1 \dots i_k)-{\sf g}_x(i_1\dots i_k)-\sum_{\ell=1}^{k-1}{\sf g}_x(i_1\dots i_\ell){\sf g}_{yx}(i_{\ell+1}\dots i_k).
\end{align*}
\end{lmm}

\begin{proof}
Only to apply ${\sf g}_y=({\sf g}_{yx}\otimes{\sf g}_x)\Delta$ to the word $(i_1\dots i_k)$.
\end{proof}

A model ${\sf g}$ is determined by the family of functions with the following properties.

\begin{dfn}
For a family $F=\{f^{i_1\dots i_k}:\mathbb{R}^d\to\mathbb{R}\}_{(i_1\dots i_k)\in W^{<1}}$ of functions, define
$$
\omega_{yx}^{i_1\dots i_k}(F):=f_y^{i_1\dots i_k}-f_x^{i_1\dots i_k}-\sum_{\ell=1}^{k-1}f_x^{i_1\dots i_\ell}\omega_{yx}^{i_{\ell+1}\dots i_k}(F).
$$
The family $\{f^{i_1\dots i_k}\}$ is called a \emph{seed of model}, if 
$$
\|F\|:=
\sup_{(i_1,\dots,i_k)\in W^{<1}}\sup_{x\in\mathbb{R}}|f_x^{i_1\dots i_k}|+
\sup_{(i_1,\dots,i_k)\in W^{<1}}\sup_{x,y\in\mathbb{R}^d}\frac{|\omega_{yx}^{i_1\dots i_k}(F)|}{|y-x|^{\alpha_{i_1}+\cdots+\alpha_{i_k}}}<\infty.
$$
\end{dfn}

Obviously, any model and its seed is linked by the equality $f_x^{i_1 \dots i_k}={\sf g}_x(i_1,\dots,i_k)$.

\subsection{Proof of Theorem \ref{3 thm main}}

Through this section, we fix the regularity parameters $\alpha_1,\dots,\alpha_n\in(0,1)$ and functions $f_i\in\mathcal{C}^{\alpha_i}$, $i=1,\dots,n$.

First we show the existence of a seed of model. We use the simplifying notations
$$
a_{<j-1}:=\sum_{i<j-1}a_i,\quad
a_{\ge j-1}:=\sum_{i\ge j-1}a_i
$$
for any sequence $\{a_j\}_{j=-1}^\infty$.

\begin{dfn}
For any $j\ge-1$, we inductively set
\begin{align*}
(f^i)_j:=\Delta_jf_i,\quad
(f^{i_1\dots i_k})_j:=(f^{i_1\dots i_{k-1}})_{<j-1}(f^{i_k})_j,
\end{align*}
(obviously, the latter definition has a meaning only if $j\ge1$) and define
\begin{align*}
f^{i_1\dots i_k}=\sum_j(f^{i_1\dots i_k})_j.
\end{align*}
\end{dfn}

\begin{rem}
It is easy to show that $\|(f^{i_1\dots i_k})_j\|_{L^\infty}\lesssim2^{-j\alpha_{i_k}}\|f_{i_1}\|_{\alpha_{i_1}}\cdots\|f_{i_k}\|_{\alpha_{i_k}}$. Hence the above series converges absolutely and defines an element of $L^\infty$. Moreover, we also have $f^{i_1\dots i_k}\in\mathcal{C}^{\alpha_{i_k}}$ by \cite[Lemma~2.84]{BCD11}.
\end{rem}

\begin{prp}\label{3 prp a seed exists}
$F=\{f^{i_1\dots i_k}\}$ is a seed of model on $\mathcal{W}^{<1}$. Precisely, if $(i_1\dots i_k)\in W^{<1}$ ($\Leftrightarrow$ $\alpha_{i_1}+\cdots+\alpha_{i_k}<1$), then one has the bound
$$
|\omega_{yx}^{i_1\dots i_k}(F)|\lesssim\|f_{i_1}\|_{\alpha_{i_1}}\cdots\|f_{i_k}\|_{\alpha_{i_k}}|y-x|^{\alpha_{i_1}+\cdots+\alpha_{i_k}}.
$$
\end{prp}

Before we turn to the proof, we prove some lemmas.

\begin{lmm}\label{3 lmm sum up}
Let $\{X_{yx}=\sum_{j=-1}^\infty X_{yx}^j\}_{x,y\in\mathbb{R}^d}$ be a family of absolutely convergent series. Assume that for some $C>0$ and $\alpha>0$, the bound
$$
|X_{yx}^j|\le C2^{j(\theta-\alpha)}|y-x|^\theta,\quad x,y\in\mathbb{R}^d
$$
holds for any $\theta$ in a neighborhood of $\alpha$, then one has the bound
$$
|X_{yx}|\lesssim C|y-x|^\alpha,\quad x,y\in\mathbb{R}^d.
$$
\end{lmm}

\begin{proof}
Without loss of generality, we can assume $|y-x|\le1$.
Fix $\epsilon>0$ such that the assumption holds for $\theta=\alpha\pm\epsilon$. For any $N\in\mathbb{N}$, we have the bounds
\begin{align*}
\sum_{j\le N}|X_{yx}^j|&\le C|y-x|^{\alpha+\epsilon}\sum_{j\le N}2^{j\epsilon}
\lesssim C2^{N\epsilon}|y-x|^{\alpha+\epsilon},\\
\sum_{j>N}|X_{yx}^j|&\le C|y-x|^{\alpha-\epsilon}\sum_{j>N}2^{-j\epsilon}
\lesssim C2^{-N\epsilon}|y-x|^{\alpha-\epsilon}.
\end{align*}
Since $|y-x|\le1$, we can choose a large $N$ such that $2^N|y-x|\simeq1$.
\end{proof}

Next we show the useful recursive formula for $\omega_{yx}^{i_1\dots i_k}(F)$. We omit the proof because it is an easy induction.

\begin{lmm}\label{3 lmm formulas}
Define
\begin{align*}
(\omega_{yx}^{i_1\dots i_k})_j
&:=(f^{i_1\dots i_k})_j(y)-(f^{i_1\dots i_k})_j(x)-\sum_{\ell=1}^{k-1}f^{i_1\dots i_{\ell}}(x)(\omega_{yx}^{i_{\ell+1}\dots i_k})_j,\\
(C_x^{i_1\dots i_k})_j&:=(f^{i_1\dots i_k})_j(x)-\sum_{\ell=1}^{k-1}(f^{i_1\dots i_\ell})(x)(C_x^{i_{\ell+1}\dots i_k})_j.
\end{align*}
Then one has the following formulas.
\begin{enumerate}
\item $(\omega_{yx}^{i})_j=\Delta_jf_i(y)-\Delta_jf_i(x)$, and for $k\ge2$,
$$
(\omega_{yx}^{i_1\dots i_k})_j=(\omega_{yx}^{i_1\dots i_{k-1}})_{<j-1}(f^{i_k})_j(y)
-(C_x^{i_1\dots i_{k-1}})_{\ge j-1}(\omega_{yx}^{i_k})_j.
$$
\item $(C_x^i)_j=\Delta_jf_i(x)$, and for $k\ge2$,
$$
(C_x^{i_1\dots i_k})_j=-(C_x^{i_1\dots i_{k-1}})_{\ge j-1}(f^{i_k})_j(x).
$$
\end{enumerate}
\end{lmm}

\begin{proof}[Proof of Proposition \ref{3 prp a seed exists}]
Since $\omega_{yx}^{i_1\dots i_k}(F)=\sum_j(\omega_{yx}^{i_1\dots i_k})_j$, we apply Lemma \ref{3 lmm sum up}.
First we show the estimate
$$
|(C_x^{i_1\dots i_k})_j|\lesssim2^{-j(\alpha_{i_1}+\cdots+\alpha_{i_k})}\|f_{i_1}\|_{\alpha_{i_1}}\cdots\|f_{i_k}\|_{\alpha_{i_k}}.
$$
If $k=1$, it holds because $f_i\in\mathcal{C}^{\alpha_i}$. Let $k\ge 2$. If $(C_x^{i_1\dots i_{k-1}})_j$ satisfies the estimate, then $(C_x^{i_1\dots i_{k-1}})_{\ge j-1}$ satisfies the same estimate. Hence we have the estimate of $(C_x^{i_1\dots i_k})_j$ by the second formula of Lemma \ref{3 lmm formulas}.

Next we show the estimate
$$
|(\omega_{yx}^{i_1\dots i_k})_j|
\lesssim2^{j(\theta-\alpha_{i_1}-\cdots-\alpha_{i_k})}\|f_{i_1}\|_{\alpha_{i_1}}\cdots\|f_{i_k}\|_{\alpha_{i_k}}|y-x|^\theta
$$
for $\theta\in(\alpha_{i_1}+\cdots+\alpha_{i_{k-1}},1]$.
For $k=1$, it holds because $f_i\in\mathcal{C}^{\alpha_i}$. Indeed, since the differentiation $\mathcal{C}^{\alpha_i}\ni f_i\mapsto\nabla f_i\in(\mathcal{C}^{\alpha_i-1})^{d}$ is continuous \cite[Proposition~2.78]{BCD11}, we have
\begin{align*}
|(\omega_{yx}^i)_j|&\le2\|\Delta_jf_i\|_{L^\infty}\lesssim2^{-j\alpha_i}\|f_i\|_{\alpha_i},\\
|(\omega_{yx}^i)_j|&\le\|\Delta_j(\nabla f_i)\|_{L^\infty}|y-x|\lesssim2^{-j(\alpha_i-1)}\|f_i\|_{\alpha_i}|y-x|.
\end{align*}
By the interpolation, we have
$$
|(\omega_{yx}^i)_j|\lesssim2^{j(\theta-\alpha_i)}\|f_i\|_{\alpha_i}|y-x|^\theta
$$
for any $\theta\in[0,1]$.
Let $k\ge 2$. If $(\omega_{yx}^{i_1\dots i_{k-1}})_j$ satisfies the estimate
$$
|(\omega_{yx}^{i_1\dots i_{k-1}})_j|
\lesssim2^{j(\theta-\alpha_{i_1}-\cdots-\alpha_{i_{k-1}})}\|f_{i_1}\|_{\alpha_{i_1}}\cdots\|f_{i_{k-1}}\|_{\alpha_{i_{k-1}}}|y-x|^\theta
$$
for $\theta\in(\alpha_{i_1}+\cdots+\alpha_{i_{k-2}},1]$, then $(\omega_{yx}^{i_1\dots i_{k-1}})_{<j-1}$ satisfies the same estimate for $\theta\in(\alpha_{i_1}+\cdots+\alpha_{i_{k-1}},1]$.
Hence we have the estimate of $(\omega_{yx}^{i_1\dots i_k})_j$ by the first formula of Lemma \ref{3 lmm formulas}.
\end{proof}

Next we show that each $f^{i_1\dots i_k}$ can be replaced by the iterated paraproduct.
The following claim is a reformulation of Theorem \ref{3 thm main}.

\begin{prp}\label{3 prp paraproduct is seed}
The family of iterated paraproducts
$$
F^\prec=\{(f_{i_1},\dots,f_{i_k})^\prec\}
$$
is a seed of model on $\mathcal{W}^{<1}$. Precisely, if $\alpha_{i_1}+\cdots+\alpha_{i_k}<1$, then one has the bound
$$
|\omega_{yx}^{i_1\dots i_k}(F^\prec)|\lesssim\|f_{i_1}\|_{\alpha_{i_1}}\cdots\|f_{i_k}\|_{\alpha_{i_k}}|y-x|^{\alpha_{i_1}+\cdots+\alpha_{i_k}}.
$$
\end{prp}

\begin{rem}
We assume $\|f_i\|_{\alpha_i}\le1$ for $i=1,\dots,n$ and show the uniform bound for such $f_1,\dots,f_n$. The general result is obtained by applying the uniform bound to normalized functions $\frac{f^1}{\|f^1\|_{\alpha_1}},\dots,\frac{f^n}{\|f^n\|_{\alpha_n}}$.
\end{rem}

We prove some lemmas. We call $\Pi=\{\tau_1,\dots,\tau_m\}$ a \emph{partition} of the word $(i_1\dots i_k)$ if there are $1=p_1<p_2<\dots<p_m<p_{m+1}=k+1$ such that $\tau_{\ell}=
(i_{p_\ell}\dots i_{p_{\ell+1}-1})$ for $\ell=1,\dots,m$.

\begin{lmm}\label{3 lmm atomic decomposition}
There are continuous functions
$$
(f_{i_1},\dots,f_{i_k})\mapsto\brc{(i_1\dots i_k)}\in\mathcal{C}^{\alpha_{i_1}+\cdots+\alpha_{i_k}}
$$
such that, one has the formula
$$
f^{i_1\dots i_k}
=\sum_{\ell=1}^{k-1}f^{i_1\dots i_{\ell}}\prec\brc{(i_{\ell+1},\dots,i_k)}+\brc{(i_1,\dots,i_k)},
$$
and moreover, one has the atomic decomposition
$$
f^{i_1\dots i_k}=\sum_{\text{$\Pi=\{\tau_1,\dots,\tau_m\}$; a partion of $(i_1 \dots i_k)$}} (\brc{\tau_1},\dots,\brc{\tau_m})^{\prec}.
$$
\end{lmm}

\begin{proof}
First formula is a consequence of the Bailleul-Hoshino's result (Theorem \ref{2 thm BH result 1}). Second formula is obtained recursively.
\end{proof}

\begin{lmm}\label{3 lmm partition}
We have the formula
$$
\omega_{yx}^{i_1\dots i_k}(F)=\sum_{\Pi=\{\tau_1,\dots,\tau_m\}}
\omega_{yx}^\prec(\brc{\tau_1},\dots,\brc{\tau_m}).
$$
\end{lmm}

\begin{proof}
We prove the formula by an induction on $k$. For a single word $(i)$, since $f^i=\brc{(i)}$, we have $\omega_{yx}^i(F)=\omega_{yx}^\pl(\brc{(i)})=f^i(y)-f^i(x)$.
Let $k\ge2$. By definition,
\begin{align*}
\omega_{yx}^\prec(\brc{\tau_1},\dots,\brc{\tau_m})
&=(\brc{\tau_1},\dots,\brc{\tau_m})^\prec(y)-(\brc{\tau_1},\dots,\brc{\tau_m})^\prec(x)\\
&\quad-\sum_{\ell=1}^{m-1}(\brc{\tau_1},\dots,\brc{\tau_\ell})^\prec(x)
\omega_{yx}^\prec(\brc{\tau_{\ell+1}},\dots,\brc{\tau_m}).
\end{align*}
We obtain the formula by summing them over all partitions $\Pi$. Indeed, for the paraproducts $(\brc{\tau_1},\dots,\brc{\tau_\ell})^\prec$, the sum is equal to $f^{i_1\dots i_{p_{\ell+1}-1}}$ because of the atomic decomposition (Lemma \ref{3 lmm atomic decomposition}). For the difference terms $\omega_{yx}^\prec(\brc{\tau_{\ell+1}},\dots,\brc{\tau_m})$, its sum is equal to $\omega_{yx}^{i_{p_{\ell+1}}\dots i_k}(F)$ by the inductive assumption. Hence we have
\begin{align*}
\sum_\Pi\omega_{yx}^\prec(\brc{\tau_1},\dots,\brc{\tau_m})
&=f^{i_1\dots i_k}(y)-f^{i_1\dots i_k}(x)-\sum_{\ell=1}^{k-1}f^{i_1\dots i_\ell}(x)
\omega_{yx}^{i_{\ell+1}\dots i_k}(F)\\
&=\omega_{yx}^{i_1\dots i_k}(F).
\end{align*}
\end{proof}

We turn to the proof of the local behavior of paraproduct.

\begin{proof}[Proof of Theorem \ref{3 thm main}]
We prove the bound by an induction on the number of components of $\omega_{yx}^\pl$.
By Lemma \ref{3 lmm partition},
$$
\omega_{yx}^\pl(f_{i_1},\dots,f_{i_k})
=\omega_{yx}^{i_1\dots i_k}(F)-\sum_{\Pi=\{\tau_1,\dots,\tau_m\};m<k}\omega_{yx}^\pl(\brc{\tau_1},\dots,\brc{\tau_m}).
$$
Since Theorem \ref{3 thm main} holds for $(k-1)$-components case, we have
$$
|\omega_{yx}^\prec(\brc{\tau_1},\dots,\brc{\tau_m})|
\lesssim|y-x|^{|\tau_1|+\cdots+|\tau_m|}=|y-x|^{\alpha_{i_1}+\cdots+\alpha_{i_k}}.
$$
By Proposition \ref{3 prp a seed exists}, we have
$$
|\omega_{yx}^\pl(f_{i_1},\dots,f_{i_k})|\lesssim|y-x|^{\alpha_{i_1}+\cdots+\alpha_{i_k}},
$$
where the implicit constant is uniform over $\|f_{i_1}\|_{\alpha_{i_1}},\dots,\|f_{i_k}\|_{\alpha_{i_k}}\le1$.
\end{proof}

\subsection{Conclusions}

We define the canonical model on the word Hopf algebra $\mathcal{W}^{<1}$.

\begin{thm}\label{3 thm canonical model}
Let $\alpha_1,\dots,\alpha_n\in(0,1)$ and $f_i\in\mathcal{C}^{\alpha_i}$, $i=1,\dots,n$. Let $\mathcal{W}^{<1}$ be the concrete regularity structure generated by alphabets $\{1,\dots,n\}$ with homogeneities $|i|=\alpha_i$. Then the function ${\sf g}:\mathbb{R}^d\to G=\chr(\mathcal{W}^{<1})$ defined by
$$
{\sf g}_x(i_1\dots i_k)=(f_{i_1},\dots,f_{i_k})^\pl
$$
is a model on $\mathcal{W}^{<1}$. Moreover, the mapping $(f_1,\dots,f_n)\mapsto{\sf g}$ is continuous.
\end{thm}

\begin{proof}
Equivalent formulation to Proposition \ref{3 prp paraproduct is seed}.
\end{proof}

We define a canonical modelled distribution on such a model.

\begin{thm}\label{3 thm canonical md}
Consider the setting of Theorem \ref{3 thm canonical model}.
Let $\beta\in(0,1)$ and $g\in\mathcal{C}^\beta$. If $\alpha_1+\cdots+\alpha_n+\beta<1$, then the function $\bsg:\mathbb{R}^d\to\mathcal{W}^{<1}$ defined by
$$
\bsg(x)=\sum_{k=0}^n(g,f_1,\dots,f_k)^\pl(x)((k+1)\dots n)
$$
belongs to the space $\mathcal{D}^{\alpha_1+\cdots+\alpha_n+\beta}({\sf g})$. Moreover, the mapping $(f_1,\dots,f_n,g)\mapsto({\sf g},\bsg)$ is continuous.
\end{thm}

\begin{proof}
By definition, $\hat{\sf g}_{yx}((k+1)\dots n)=\sum_{\ell=k}^n{\sf g}_{yx}((k+1)\dots\ell)((\ell+1)\dots n)$. Since ${\sf g}_{yx}((k+1)\dots\ell)=\omega_{yx}^\pl(f_{k+1},\dots,f_\ell)$, we have
\begin{align*}
&\bsg(y)-\hat{\sf g}_{yx}\bsg(x)\\
&=\sum_{\ell=0}^n\left\{(g,f_1,\dots,f_\ell)^\pl(y)-\sum_{k=0}^\ell(g,f_1,\dots,f_k)\omega_{yx}^\pl(f_{k+1},\dots,f_\ell)\right\}((\ell+1)\dots n)\\
&=\sum_{\ell=0}^n\omega_{yx}^\pl(g,f_1,\dots,f_\ell)((\ell+1)\dots n).
\end{align*}
By Theorem \ref{3 thm main}, we have that $\bsg\in\mathcal{D}^{\alpha_1+\cdots+\alpha_n+\beta}({\sf g})$.
\end{proof}

\section{Commutator}\label{4 section commutator}

Finally we show how to apply Bailleul-Hoshino's result to the commutator estimate.

\subsection{Commutator estimate}

The commutator in \cite{GIP15} is defined by
$$
C(f,g,h)=(f\pl g)\rs h-f(g\rs h).
$$
However, in this paper, we consider the commutator with respect to the bilinear operator $\pge =\rs +\pr $ instead of the resonant $\rs $. Such operators are sufficient in applications.

\begin{dfn}
For any smooth functions $g,f_1,f_2,\dots,f_n,\xi$ on $\mathbb{R}^d$, we define
\begin{align*}
{\sf C}(g,\xi)&:=g\succeq\xi\ (=g\rs\xi+g\pr\xi),\\
{\sf C}(g,f_1,\xi)&:={\sf C}(g\pl f_1,\xi)-g{\sf C}(f_1,\xi),\\
{\sf C}(g,f_1,f_2,\dots,f_n,\xi)&:={\sf C}(g\pl f_1,f_2,\dots,f_n,\xi)-g{\sf C}(f_1,f_2,\dots,f_n,\xi).
\end{align*}
\end{dfn}

Our aim is to show the following commutator estimate. Let $\mathcal{C}_0^\alpha$ be the closure of the smooth functions in $\mathcal{C}^\alpha$.

\begin{thm}\label{4 thm commutator estimate}
Let $\beta,\alpha_1,\dots,\alpha_n\in(0,1)$ and $\gamma<0$ be such that
\begin{align*}
&\beta+\alpha_1+\cdots+\alpha_n<1,\\
&\alpha_1+\cdots+\alpha_n+\gamma<0
<\beta+\alpha_1+\cdots+\alpha_n+\gamma.
\end{align*}
Then there exists a unique multilinear continuous operator
$$
\tilde{\sf C}:\mathcal{C}_0^\beta\times\mathcal{C}_0^{\alpha_1}\times\cdots\mathcal{C}_0^{\alpha_n}\times\mathcal{C}_0^\gamma
\to\mathcal{C}_0^{\beta+\alpha_1+\cdots+\alpha_n+\gamma}
$$
such that,
$$
\tilde{\sf C}(g,f_1,\dots,f_n,\xi)
={\sf C}(g,f_1,\dots,f_n,\xi)
$$
for any smooth inputs $(g,f_1,\dots,f_n,\xi)$.
\end{thm}

In this paper, we show only the existence of the \emph{continuous} map $\tilde{\sf C}$. The uniqueness of $\tilde{\sf C}$ and its multilinearity follows by the denseness argument.

\begin{rem}
Reading the proof carefully, we can see that the operator norm of $\tilde{\sf C}$ is locally uniform over regularity parameters.
Then by the similar argument to \cite[Lemma~2.4]{GIP15}, we also obtain the unique multilinear continuous operator
$$
\tilde{\sf C}:\mathcal{C}^\beta\times\mathcal{C}^{\alpha_1}\times\cdots\mathcal{C}^{\alpha_n}\times\mathcal{C}^\gamma
\to\mathcal{C}^{\beta+\alpha_1+\cdots+\alpha_n+\gamma},
$$
such that $\tilde{\sf C}={\sf C}$ on any smooth inputs. However, we do not prove it in this paper, because the estimate on the space $\mathcal{C}_0^\alpha$ is sufficient in applications.
\end{rem}

\subsection{Proof of Theorem \ref{4 thm commutator estimate}}

First we introduce another concrete regularity structure.
Let $\mathcal{W}^{<1}$ be the graded word Hopf algebra generated by alphabets $\{1,\dots,n\}$ with homogeneities $|i|=\alpha_i$. We introduce another alphabet $\Xi$ with the homogeneity $|\Xi|=\gamma$, which represents the distribution $\xi\in\mathcal{C}^\gamma$.

\begin{dfn}
Let $T$ be a linear space spanned by abstract variables
$$
\Xi,\quad
(k\dots n)\Xi\quad (k=1,\dots,n),
$$
with homogeneities $|(k\dots n)\Xi|=\alpha_k+\cdots+\alpha_n+\gamma$. The coproduct $\Delta:T\to T\otimes\mathcal{W}^{<1}$ is defined by
\begin{align*}
\Delta\Xi:=\Xi\otimes\mathbf{1},\quad
\Delta(k\dots n)\Xi:=\sum_{\ell=k-1}^{n}((\ell+1)\dots n)\Xi\otimes(k\dots\ell).
\end{align*}
\end{dfn}

It is easy to show the following fact.

\begin{lmm}
$(T,\Delta)$ is a comodule over $\mathcal{W}^{<1}$. Thus $(\mathcal{W}^{<1},T)$ is a concrete regularity structure.
\end{lmm}

We consider a model $({\sf g},{\sf\Pi})$ on $(\mathcal{W}^{<1},T)$.
For given $f_i\in\mathcal{C}^{\alpha_i}$, $i=1,\dots,n$, let ${\sf g}$ be the canonical model on $\mathcal{W}^{<1}$ defined by
$$
{\sf g}_x(i_1\dots i_k)=(f_{i_1},\dots,f_{i_k})^\pl.
$$
By Theorem \ref{3 thm canonical model}, we have $\|{\sf g}\|<\infty$.
Then by Proposition \ref{2 prp brc to model}, we can define a linear map ${\sf\Pi}:T\to\mathcal{S}'(\mathbb{R}^d)$ for any given data $\brc{\tau}$ for the bases $\tau$ with negative homogeneities, i.e., $\Xi$ and $(k\dots n)\Xi$ for $k=1,\dots,n$.

\begin{lmm}
Let $\xi\in\mathcal{C}^\gamma$.
We define a model ${\sf Z}^0=({\sf g},{\sf \Pi}^0)$ on $(\mathcal{W}^{<1},T)$ by
$$
\brc{\Xi}^{{\sf Z}^0}=\xi,\quad
\brc{(k\dots n)\Xi}^{{\sf Z}^0}=0.
$$
Then the map
\begin{align*}
(f_1,\dots,f_n,\xi)\mapsto{\sf Z}^0
\end{align*}
is continuous.
\end{lmm}

If all inputs $(f_1,\dots,f_n,\xi)$ are smooth, we can define another model.

\begin{dfn}
If all of $f_1,\dots,f_n,\xi$ are smooth, then we define a smooth model ${\sf Z}^s=({\sf g},{\sf \Pi}^s)$ on $(\mathcal{W}^{<1},T)$ by
$$
{\sf \Pi}^s\Xi=\xi,\quad
{\sf\Pi}^s(k\dots n)\Xi=(f_k,\dots,f_n)^\pl\xi.
$$
\end{dfn}

We can check the bound $\|{\sf\Pi}^s\|^{\sf g}<\infty$ because ${\sf\Pi}^s$ maps $T$ into smooth functions.

\begin{lmm}\label{4 lmm Pi^s=pro}
If all of $f_1,\dots,f_n,\xi$ are smooth, then we have the equalities
$$
(({\sf\Pi}^s)_x^{\sf g}\Xi)(x)=\xi(x),\quad
(({\sf\Pi}^s)_x^{\sf g}(k\dots n)\Xi)(x)=0.
$$
\end{lmm}

\begin{proof}
Since ${\sf\Pi}^s(\tau\Xi)(x)={\sf g}_x(\tau)\xi(x)$ for each word $\tau$,
\begin{align*}
({\sf\Pi}^s)_x^{\sf g}(\tau\Xi)(x)&=({\sf\Pi}^s\otimes{\sf g}_x^{-1})\Delta(\tau\Xi)(x)\\
&=({\sf g}_x\otimes{\sf g}_x^{-1})\Delta\tau(x)\xi(x)=\mathbf{1}^*(\tau)\xi(x).
\end{align*}
\end{proof}

Since $({\sf g},{\sf\Pi}^s)$ and $({\sf g},{\sf\Pi}^0)$ are models, $({\sf g},\ovl{\sf\Pi}):=({\sf g},{\sf\Pi}^s-{\sf\Pi}^0)$ is also a model.

\begin{lmm}\label{4 lmm ovlPi=C}
If all of $f_1,\dots,f_n,\xi$ are smooth, then we have the equalities
$$
(\ovl{\sf\Pi}_x^{\sf g}\Xi)(x)=0,\quad
(\ovl{\sf\Pi}_x^{\sf g}(k\dots n)\Xi)(x)={\sf C}(f_k,\dots,f_n,\xi)(x).
$$
\end{lmm}

\begin{proof}
Recall that $\ovl{\sf\Pi}_x^{\sf g}=(\ovl{\sf\Pi}\otimes{\sf g}_x^{-1})\Delta$.
We have the first equality because $\Delta\Xi=\Xi\otimes\mathbf{1}$ and $\ovl{\sf\Pi}\Xi={\sf\Pi}^s\Xi-{\sf\Pi}^0\Xi=\xi-\xi=0$. The second equality is proved by an induction on $k$. By the formula $\ovl{\sf\Pi}=(\ovl{\sf\Pi}_x^{\sf g}\otimes{\sf g}_x)\Delta$, we have
\begin{align*}
\ovl{\sf\Pi}_x^{\sf g}(k\dots n)\Xi
=\ovl{\sf\Pi}(k\dots n)\Xi-\sum_{\ell=k}^n{\sf g}_x(k\dots \ell)\ovl{\sf\Pi}_x^{\sf g}((\ell+1)\dots n)\Xi.
\end{align*}
We calculate the term $\ovl{\sf\Pi}(k\dots n)\Xi$. Since the Bailleul-Hoshino's result (Theorem \ref{2 thm BH result 1}) shows
\begin{align*}
{\sf\Pi}^0(k\dots n)\Xi&=\sum_{\ell=k}^n{\sf g}(k\dots \ell)\pl\brc{((\ell+1)\dots n)\Xi}^{{\sf Z}^0}+\brc{(k\dots n)\Xi}^{{\sf Z}^0}\\
&=(f_k,\dots,f_n)^\pl\pl\xi,
\end{align*}
thus we have
\begin{align*}
\ovl{\sf\Pi}(k\dots n)\Xi&=(f_k,\dots,f_n)^\pl\xi-(f_k,\dots,f_n)^\pl\pl\xi\\
&={\sf C}((f_k,\dots,f_n)^\pl,\xi).
\end{align*}
If $k=n$, we have $(\ovl{\sf\Pi}_x^{\sf g}(n)\Xi)(x)={\sf C}(f_n,\xi)(x)$. If the required formula holds for $m+1\le k\le n$, then we have
\begin{align*}
(\ovl{\sf\Pi}_x^{\sf g}(m\dots n)\Xi)(x)
&={\sf C}((f_m,\dots,f_n)^\pl,\xi)(x)-\sum_{\ell=m}^{n-1}(f_m,\dots,f_\ell)^\pl(x){\sf C}(f_{\ell+1},\dots,f_n,\xi)(x)\\
&={\sf C}(f_m,\dots,f_n,\xi)(x).
\end{align*}
The last equality follows by the definition of commutator.
\end{proof}

Finally we turn to the proof of commutator estimate.

\begin{proof}[Proof of Theorem \ref{4 thm commutator estimate}]
By Theorem \ref{3 thm canonical md}, a $\mathcal{W}^{<1}$-valued function
$$
\bsg(x)=\sum_{k=0}^n(g,f_1,\dots,f_k)^\prec(x)((k+1)\dots n)
$$
belongs to $\mathcal{D}^{\beta+\alpha_1+\cdots+\alpha_n}(\mathcal{W}^{<1};{\sf g})$. It is easy to show that a $T$-valued function
$$
(\bsg\Xi)(x)=\sum_{k=0}^n(g,f_1,\dots,f_k)^\prec(x)((k+1)\dots n)\Xi
$$
belongs to $\mathcal{D}^{\beta+\alpha_1+\cdots+\alpha_n+\gamma}(T;{\sf g})$. (Here we specify the range space to distinguish two different classes.) This is a consequence of \cite[Theorem~4.7]{Hai14}, but it is not difficult to show directly it because
\begin{align*}
\|(\bsg\Xi)(y)-\hat{\sf g}_{yx}(\bsg\Xi)(x)\|_{\alpha_{k+1}+\cdots+\alpha_n+\gamma}
&=\|\bsg(y)-\hat{\sf g}_{yx}\bsg(x)\|_{\alpha_{k+1}+\cdots+\alpha_n}\\
&\le\tri\bsg\tri_{\beta+\alpha_1+\cdots+\alpha_n}|y-x|^{\beta+\alpha_1+\cdots+\alpha_k}.
\end{align*}
Applying the Bailleul-Hoshino's reconstruction theorem (Theorem \ref{2 thm reconst}), we have
$$
{\sf R}^{{\sf Z}^0}(\bsg\Xi)=(g,f_1,\dots,f_n)^\pl\pl\xi+\brc{\bsg\Xi}^{{\sf Z}^0},
$$
and the map
\begin{align*}
(g,f_1,\dots,f_n,\xi)\mapsto
\brc{\bsg\Xi}^{{\sf Z}^0}
\end{align*}
is continuous. It turns out that this is the required map $\tilde{\sf C}$.
It remains to show that
$$
\brc{\bsg\Xi}^{{\sf Z}^0}={\sf C}(g,f_1,\dots,f_n,\xi)
$$
for any smooth inputs $(g,f_1,\dots,f_n,\xi)$.
Let ${\sf\Pi}^s$ be the smooth model as above. Since ${\sf R}^{\ovl{\sf Z}}={\sf R}^{{\sf Z}^s}-{\sf R}^{{\sf Z}^0}$, we have
$$
{\sf R}^{\ovl{\sf Z}}(\bsg\Xi)={\sf R}^{{\sf Z}^s}(\bsg\Xi)-{\sf R}^{{\sf Z}^0}(\bsg\Xi).
$$
Note that, for any model ${\sf Z}=({\sf g},{\sf\Pi})$ such that ${\sf\Pi}$ maps $T$ into smooth functions, we have
$$
{\sf R}^{\sf Z}(\bsg\Xi)(x)=({\sf\Pi}_x^{\sf g}(\bsg\Xi)(x))(x)
$$
by the uniqueness of the reconstruction operator (see \cite[Remark~3.15]{Hai14}).
For ${\sf Z}={\sf Z}^s$, by Lemma \ref{4 lmm Pi^s=pro} we have
$$
{\sf R}^{{\sf Z}^s}(\bsg\Xi)=(g,f_1,\dots,f_n)^\pl\xi.
$$
For ${\sf Z}=\ovl{\sf Z}$, by Lemma \ref{4 lmm ovlPi=C} we have
$$
{\sf R}^{\ovl{\sf Z}}(\bsg\Xi)=\sum_{k=0}^{n-1}(g,f_1,\dots,f_k)^\pl{\sf C}(f_{k+1},\dots,f_n,\xi).
$$
Hence we have
\begin{align*}
\brc{\bsg\Xi}^{{\sf Z}^0}&={\sf C}((g,f_1,\dots,f_n)^\pl,\xi)-\sum_{k=0}^{n-1}(g,f_1,\dots,f_k)^\pl{\sf C}(f_{k+1},\dots,f_n,\xi)\\
&={\sf C}(g,f_1,\dots,f_n,\xi).
\end{align*}
\end{proof}

\end{document}